\documentclass[12pt]{article}
\usepackage[utf8]{inputenc}
\usepackage[english]{babel}
\usepackage{amsthm}
\usepackage{geometry}
\usepackage{amstext}
\usepackage{amssymb}
\usepackage{amsmath}
\usepackage{setspace}
\usepackage{cmap}
\usepackage{authblk}
\usepackage{hyperref}
\newtheorem{thm}{Theorem}
\newtheorem{lem}{Lemma}
\newtheorem{cor}{Corollary}
\theoremstyle{definition}
\newtheorem{rmk}{Remark}
\begin{document}
\newcommand{\Var}{\operatorname{Var}}
\author[1]{Y.~Kirpicheva}
\author[1]{A.~Shklyaev \thanks{alexander.shklyaev@math.msu.ru}} %
\title{Branching Process in a Varying Environment: How to Grow Like the Product of Means} 
\affil[1]{Lomonosov Moscow State University, Moscow, Russia}
\date{}
\maketitle
\begin{abstract}
Consider a branching process $\{Z_n\}$ in a varying environment. Let $\{W_n\}$ be the natural martingale $Z_n/{\bf E}Z_n$. It converges to some random variable $W$ as $n\to\infty$. An important problem is to show that ${\bf P}(W>0)$ equals the survival probability, so that $Z_n$ is either $0$ or of the order ${\bf E}Z_n$. We find a new kind of sufficient conditions, applicable to branching processes in a random environment. An important property of our estimates is that we don't necessary assume that ${\bf E}X_{i,1}^2$ are finite for every $i$.    
\end{abstract}
\section{Introduction}
Let $P_i=\{p_{i,j}, j\ge 0\}$ be a sequence of distributions on $\mathbb{N}_0=\mathbb{N}\cup \{0\}$, and let $(X_{i,j}, i>0)$ be independent sequences of independent random variables, $X_{i,j}\sim P_i$, $i\ge 1$. Define $\{Z_n\}$ by $Z_0=1$, 
$$
Z_{n+1} = \sum_{i=1}^{Z_n} X_{n+1,i}, n\ge 0.
$$
The process $\{Z_n\}$ is called a branching process in a varying environment (BPVE). 

Let $\mathcal{P}$ be the space of all probability distributions on $\mathbb{N}_0$:
$$\mathcal{P}=\left\{p_{i}: p_i\ge 0, \sum_{i=0}^{\infty} p_i\right\}.$$ 
Assume that ${\bf Q}$ is some distribution on the space $\mathcal{P}^{\infty}$. Let $\mathrm{R}=(R_i, i\ge 1)$ be a sequence with distribution ${\bf Q}$.
A process $\{Z_n\}$ is called a branching process in a random environment (BPRE) if for every fixed $\mathrm{R}=(r_1,r_2,\dotsc)$ the sequence $\{Z_n\}$ is a BPVE in the environment $(r_1,r_2,\dotsc)$.

The BPRE model has been studied in detail, see \cite{VatKerst} for a variety of results. The case of BPVE is has not been studied enough. The pioneering works of~\cite{Jagers} and~\cite{Agresti} give some basic results. 

We are interested in the behaviour of the natural martingale $W_n = Z_{n}/{\bf E}Z_n$, $n\ge 1$. In the important recent work~\cite{Kersting} G. Kersting gives a general criterion for $\{W_n\}$ to converge. However, Kersting's work have two important disadvantages. The key problem is condition (A):
\begin{equation}
\label{KestingA}
\sup_n \frac{{\bf E}(X_{n,1}^2; X_{n,1}\ge 2)}{{\bf E}(X_{n,1}|X_{n,1}\ge 1) {\bf E}(X_{n,1}; X_{n,1}\ge 2)}<+\infty.
\end{equation}
This condition is quite restrictive for BPRE. The second problem is that ${\bf E}X_n^2$ must be finite for every $n$. 

So, the powerful techniques for BPVE, described in~\cite{Kersting}, are too restrictive for BPRE. On the other hand, standard techniques for BPRE usually assume some stationarity of the environment. We want to find a general approach for BPVE, that works for BPRE too, especially in the case of a non-stationary environment. For example, for BPRE with cooling (see~\cite{Korshunov}) we have a random, but non-stationary environment. 

We are working with the following question -- does the martingale $\{W_n\}$ converge a.s and in $L^1$ to some r.v. $W$ and is it true that 
$$
{\bf P}(W>0) = {\bf P}(Z_n>0, n>0)?
$$
The key question is the last one. Assume that ${\bf E}X_{i,1}^2$ are finite for every $i$. Then we assume that
\begin{eqnarray*}
{\bf P}(Z_n\to \infty) + {\bf P}(\exists n: Z_n=0) = 1,\\
\liminf_{l\to \infty} \sum_{i=0}^{\infty} 
\frac{\Var X_{i+l,1}}{\left({\bf E}X_{i+l,1}\right)^2} \prod_{j=l}^{l+i-1} \frac{1}{{\bf E}X_{j,1}} < +\infty.
\end{eqnarray*}
The first condition is standard (see~\cite{Jagers}) and weak. The second is much less restrictive than~(\ref{KestingA}). The corresponding Theorem~\ref{BPVETh} is stated and proved in Subsection~\ref{VarianceSec}.

If we only assume that ${\bf E}X_{i,1} \phi(X_{i,1})$ is finite, where $\phi(x)=o(x)$, $x\to\infty$, then we give another assumption~ (\ref{Ass2ForPsi}) instead of (\ref{Ass2}). The corresponding results are stated in Subsection~\ref{GeneralSec}.

In Section~\ref{Applications} we give some applications for BPRE. We show that our conditions for the case of supercritical BPRE are only slightly weaker than the best possible conditions of~\cite{Tanny}. For a critical BPRE under the condition ${\bf E}_{\boldsymbol\eta} X_{1,1}^2<+\infty$ our general result gives the same sufficient conditions as in the fundamental work~\cite{AGKV}. However, we obtain new results for a critical BPRE in the case ${\bf E}_{\boldsymbol \eta} X_{1,1}^2=+\infty$. 
\section{Main Results}
\subsection{Variance Condition on BPVE}
\label{VarianceSec}
\begin{thm}
\label{BPVETh}
    Let $\{Z_n\}$ be a BPVE, ${\bf E}X_{i,1}$ be positive and $\Var X_{i,1}$ be finite for every $i$. Let 
    $\xi_i = \log {\bf E}X_{i,1}$, $\zeta_i = \Var X_{i,1} e^{-2\xi_i}$, $i>0$. Let $S_n=\xi_1+\dotsb+\xi_n$. Assume that 
    \begin{eqnarray}  
    \label{Ass1} {\bf P}\left(Z_n \rightarrow \infty \right) + {\bf P}\left(\exists n:  Z_n = 0 \right) = 1,\\
    \label{Ass2} \liminf_{l \to \infty} a_l<+\infty,\ a_l:=\sum_{i = 0}^{\infty} \zeta_{i+l} e^{-(S_{i+l} - S_l)}. 
    \end{eqnarray}
    Then the martingale $W_n=Z_n/\exp(S_n)$, $n\ge 1$, converges in $L^2$ to some random variable $W$ and 
    \begin{equation}
    \label{WandZn}
    {\bf P}(W>0)={\bf P}(Z_n>0,\ n>0)>0.
    \end{equation}
\end{thm}
\begin{rmk}
\label{JagersTh}
Theorem 1 of~\cite{Jagers} says that~(\ref{Ass1}) is equivalent to 
$$
\sum_{i=1}^{\infty} \left(1-{\bf P}(X_{i,1}=1)\right)=\infty
$$
if ${\bf P}(X_{i,1}=0)<1$ for every $i$.
\end{rmk}
\begin{proof}[Proof of Theorem~\ref{BPVETh}]
Note that~(\ref{Ass2}) implies that $a_i<+\infty$ for every $i$, since 
$$
a_{l} = a_{l+1} e^{-\xi_l} + \zeta_{l}.
$$
The convergence of $\{W_n\}$ in $L^2({\bf P})$ is well-known (see Theorem 2 of Kersting). Hovewer, we'll prove it for the reader's convenience: for any natural $m$, $k$ we have
\begin{eqnarray}
\label{L2Est}
{\bf E}\left(\left.\left|\frac{Z_m}{e^{S_m}} - \frac{Z_{m-1}}{e^{S_{m-1}}} \right|^2\right|Z_0=k\right) = e^{-2S_m} {\bf E}(\operatorname{Var}(Z_m|Z_{m-1})|Z_0=k) = \\
e^{-2S_{m-1}} \zeta_m {\bf E}\left(Z_{m-1}|Z_0=k\right) = k \zeta_m e^{-S_{m-1}}.\nonumber
\end{eqnarray}
The standard martingale trick gives us
\begin{eqnarray}
\label{MartingaleL2}
\operatorname{cov}(W_{n+m}-W_{n+m-1},W_{n+m-1}-W_n)=\\
{\bf E}\left({\bf E}\left(\left.\left(W_{n+m}-W_{n+m-1}\right)\right|W_{n+m-1}\right) (W_{n+m-1} - W_n)\right)=0,\ n,m\in \mathbb{N},\nonumber
\end{eqnarray}
so, for any natural $m,n$
\begin{equation}
\label{L2Ineq}
{\bf E}\left(\left.(W_{n+m}-W_n)^2\right|Z_0=k\right) = \sum_{i=0}^{m-1} {\bf E}\left(\left.(W_{n+i+1}-W_{n+i})^2\right|Z_0=k\right)\le 
k a_n.
\end{equation}
Thus, the condition $a_1<+\infty$ leads us to the convergence of $\{W_n\}$ in $L^2({\bf P})$ by the Cauchy convergence test: for any positive $\varepsilon$, any natural $k$, all large enough $m$ and any $n$ we have
$$
{\bf E}\left(\left.\left|W_{m+n}- W_m \right|^2\right|Z_0=k\right)\le k \sum_{i=m+1}^{m+n} \zeta_m e^{-S_{m-1}}<\varepsilon.$$
Since ${\bf E}W_n={\bf E}W_0$ we have ${\bf E}W={\bf E}W_0$. Thus, ${\bf P}(W>0)>0$. 

The main part of Theorem~\ref{BPVETh} is~(\ref{WandZn}). 
Let us introduce the events
$$
D_Z:= \{Z_i>0,\ i>0\},\quad D_W:= \{W>0\}.
$$
We need to show that ${\bf P}(D_Z\setminus D_W)=0$. 

By~(\ref{Ass2}) for some positive $M$ 
$$
a_l = \sum_{j=0}^{\infty} \zeta_{l+j} e^{S_{l+j}-S_l} < M.
$$
Let $K>8M$ be a natural number and let $\tau_0=\nu_0=0$,
\begin{eqnarray*}
\tau_l = \min\left\{j>\nu_{l-1}: a_j < M,\ Z_{j}\ge K\right\},\\
\nu_l = \min\left\{i>\tau_l: \left(\frac{Z_i}{e^{S_i - S_{\tau_l}}} < \frac{Z_{\tau_l}}{2} \right)\right\},\ l\ge 1.
\end{eqnarray*}
Note that on $D_Z$ the random variable $\tau_l$ is a.s. finite if $\nu_{l-1}<+\infty$ due to~(\ref{Ass1}) and~(\ref{Ass2}).
 If $N=\max\{l: \nu_l<+\infty\}<+\infty$, then
$$
\frac{Z_{\tau_{N+1}+i}}{e^{S_{\tau_{N+1}+i}}}\ge \frac{Z_{\tau_{N+1}}}{2e^{S_{\tau_{N+1}}}}, i>0,
$$
and $W$ is positive. So, we need to prove that only a finite number of $\{\nu_l\}$ occurs a.s.

Thus, for any $l,m\in \mathbb{N}$ and any $k\ge K$ from~(\ref{L2Est}) we have 
\begin{eqnarray*}
{\bf P}\left(\left.\exists j>\tau_l: \frac{Z_j}{e^{S_j-S_{\tau_l}}}<\frac{k}{2}\right|Z_{\tau_l}=k,\tau_l=m\right) =\\ 
{\bf P}\left(\left. \max_{j> 0} \left| \frac{Z_{\tau_{l}+j}}{e^{S_{\tau_{l}+j}-S_{\tau_l}}} - Z_{\tau_{l}} \right| \geq \frac{k}{2}\right|Z_{\tau_l}=k, \tau_l=m\right)\le \\
{\bf P}\left(\left.\sum_{j=1}^{\infty} 
 \left( \frac{Z_{m+j}}{e^{S_{m+j}-S_m}} - \frac{Z_{m+j-1}}{e^{S_{m+j-1}-S_m}} \right)^2 > \frac{k^2}{4}
\right| Z_{m} = k \right)<
\frac{4a_{m}k}{k^2}<\frac{4M}{K}<\frac12.
\end{eqnarray*}
Therefore, for any natural $n$
$$
{\bf P}(\nu_n<+\infty) = 
{\bf E}\left({\bf P}(\nu_n<+\infty|Z_{\tau_{n-1}}, \tau_{n-1}) I_{\nu_{n-1}<+\infty}\right) < \frac{1}{2} {\bf P}(\nu_{n-1}<\infty)<\frac{1}{2^n}.
$$
So, $N$ is a.s. finite on the event $D_Z$. This proves the Theorem.
\end{proof}
A direct application of Theorem~\ref{BPVETh} is the following theorem.
\begin{thm}
\label{FuncTh}
Let $\{r_n\}$ be a sequence of natural numbers such that $r_n=o(n)$, $r_n\to +\infty$, $n\to\infty$. Let
$$
Y_n(t) = \frac{Z_{[r_n+(n-r_n)t]}}{\exp(S_{[r_n+(n-r_n)t]})},\ t\in [0,1].
$$
Let 
$${\bf Q}_n(A)={\bf P}(Y_n(\cdot)\in A|Z_n>0),\quad 
{\bf Q}(A) = {\bf P}(Y(\cdot)\in A),
$$
for every Borel set $A$ from the cadlag space $D[0,1]$, equipped with the Skorohod topology, where $Y(\cdot)$ is a process with a.s. constant trajectories and $Y(1)\stackrel{d}{=}W$. Then 
${\bf Q}_n\stackrel{w}{\to} {\bf Q}$ as $n\to\infty.$
\end{thm}
\begin{proof}
Let $Y(1)=W$. By Theorem~\ref{BPVETh} we have 
$$Y_n(\cdot)\to Y(\cdot),\ n\to\infty,$$ 
in $D[0,1]$, equipped with the uniform topology (thus, with the Skorokhod topology as well). Similarly, $I_{Y_n(1)>0}$ converges to $I_{Y(1)>0}$ a.s. Thus, for every continuous bounded function $\phi$ on $D[0,1]$, equipped with the Skorohod topology, we have
$$
{\bf E}\phi(Y_n(\cdot)) I_{Y_n(1)>0}\to 
{\bf E}\phi(Y(\cdot)) I_{Y(1)>0},\ n\to\infty,
$$
by the Lebesgue dominated convergence theorem. Thus, 
$$
{\bf E}_{{\bf Q}_n}\phi(Y_n(\cdot))\to 
{\bf E}_{{\bf Q}}\phi(Y(\cdot)),\ n\to\infty.
$$
This proves the Theorem. 
\end{proof}
In the case of BPRE Theorem~\ref{FuncTh} helps us to obtain a Central Limit Theorem and a Functional Limit Theorem for $\{Z_n\}$, conditioned to survive, if we know the corresponding result for $\{S_n\}$.
\subsection{General Case of BPVE}
\label{GeneralSec}
We use variances in Theorem~\ref{BPVETh} to make condition~(\ref{Ass2}) similar to conditions of \cite{Kersting}. However, instead of $L^2$-norm we can consider other distances.

Let $\psi(x)=x\phi(x)$, where $\psi$ is an even smooth function on $\mathbb{R}$, concave on $[0,\infty)$ and $\phi(x)$ is an increasing function on $[0,\infty)$. Then by Theorem 1.1 of \cite{Pinelis}
as $Z_0=k$, $k\in \mathbb{N}$, we have
\begin{eqnarray*}
{\bf E}\psi(W_m - W_{m-1}) = 
{\bf E}\left({\bf E}\left(\left.\psi\left(\sum_{i=1}^{Z_{m-1}} \left(X_{m,i}-e^{\xi_m}\right) e^{-S_m}\right)\right|W_{m-1}\right) \right)\le \\
C {\bf E} Z_{m-1} {\bf E}\psi\left(\left(X_{m,1}-e^{\xi_m}\right) e^{-S_m}\right) = 
C k {\bf E}\left(U_m\phi\left(U_m e^{-S_{m-1}}\right)\right),
\end{eqnarray*}
where $U_m=|X_{m,1}e^{-\xi_m}-1|$ and $C$ is some constant. Using this and Theorem 1.1 we get for any natural $m$, $n$
\begin{equation}
\label{PhiWIneq}
{\bf E}\left(\left.\psi(W_{m+n} - W_{m})\right|Y_0=k\right)\le 
C^2 k \sum_{j=1}^{n-1} {\bf E}\left(U_{m+j}\phi\left( U_{m+j} e^{-(S_{m+j-1}-S_m)}\right)\right).
\end{equation}
Let
$$
a_m^{(\psi)} = C^2 k \sum_{j=1}^{\infty} {\bf E}\left(U_{m+j}\phi\left( U_{m+j} e^{-(S_{m+j-1}-S_m)}\right)\right),\ m\in \mathbb{N}_0.
$$
\begin{thm}
\label{BPVEWeakTh}
Let $\{Z_n\}$ be a BPVE, let ${\bf E}X_{n,1}$
 be positive and finite for every $n\in\mathbb{N}$, and let $\xi_n=\log {\bf E}X_{n,1}$, $n\in \mathbb{N}$. Assume that 
$$
{\bf E}\left(U_{m+j}\phi\left(U_{m+j} e^{-(S_{m+j-1}-S_m)}\right)\right)<+\infty
$$
for any natural $m$, $j$.
Assume~(\ref{Ass1}) and 
\begin{equation}
\label{Ass2ForPsi}
\liminf_{l} a_l^{(\psi)} < +\infty.  
\end{equation}
Then $\{W_n\}$ converges a.s. and in 
$L^1$
to $W$ and
$$
{\bf P}(W>0) = {\bf P}(Z_n>0, n>0) > 0.
$$
\end{thm}
\begin{proof}
The convergence of $\{W_n\}$ in $L^1({\bf P})$ is a corollary of~\ref{PhiWIneq}. Since ${\bf E}W_n={\bf E}W_0$ for every $n$ we have ${\bf E}W>0$. The proof of statement~(\ref{WandZn}) is based only on~(\ref{Ass1}), (\ref{Ass2}), the Markov inequality and~(\ref{L2Ineq}). Using instead~(\ref{Ass2ForPsi}) and (\ref{PhiWIneq}), we obtain the Theorem.
\end{proof}
\begin{rmk}
Since $\psi$ is a Young function, we can introduce the Orlicz space $(L_{\psi}({\bf P}), ||\cdot||_{\psi})$. Under the conditions of Theorem~\ref{Ass2ForPsi}, we can conclude that the sequence $\{W_n\}$ converges to $W$ in $L_{\psi}({\bf P})$ by Theorem 1.5 of \cite{Orlicz}, since we have 
$$
{\bf E}\psi(W_n-W_m)\to 0,\ n,m\to\infty.
$$

\end{rmk}
A nice particular case is $\psi(x)=|x|^{1+\delta}$ for some $\delta\in (0,1]$. Then 
\begin{equation}
\label{Ass2L1+delta}    
\liminf_{l \to \infty} a_l^{(\delta)}<+\infty, \ a_{l}^{(\delta)}:=\sum_{j=0}^{\infty} \zeta_{l+j}^{(\delta)} e^{-\delta (S_{l+j}-S_l)},
\end{equation}
is equivalent to~(\ref{Ass2ForPsi}).

Thus, we obtain the following Corollary.
\begin{cor}
\label{BPVEWeakTh}
Let $\{Z_n\}$ be a BPVE, let ${\bf E}X_{n,1}$ be positive and finite for every $n$, and let $\xi_n=\log {\bf E}X_{n,1}$, $n\ge 1$. Assume~(\ref{Ass1}).
Let $\zeta_i^{(\delta)}$ be finite for some positive $\delta$ and assume~(\ref{Ass2L1+delta}). Then the martingale $\{W_n\}$ converges in $L^{1+\delta}({\bf P})$ to some random variable $W$ and 
$$
{\bf P}(W>0) = {\bf P}(Z_n>0, n>0)>0.
$$
\end{cor}
This generalizes Theorem~\ref{BPVETh}.
\section{Applications}
\label{Applications}
The result can be applied to different models of branching processes in a random environment. Our results have an important advantage with respect to those of~\cite{Kersting} -- they can be applied to BPRE without any uniform boundedness conditions like~(\ref{KestingA}). 

For branching processes in a random environment (not necessarily i.i.d.) the following lemma is useful.
\begin{lem}
\label{RegeneratingEnvLemm}
Assume that $\{Z_n\}$ is a BPRE with the environment $\mathrm{R}$. If $R_i$ are independent, then~ (\ref{Ass2}) is a.s. true iff 
$$
a_l\not\stackrel{d}{\to} +\infty,\ l\to\infty.
$$
\end{lem}
\begin{rmk}
Similar facts are true for conditions~(\ref{Ass2ForPsi}) and~(\ref{Ass2L1+delta}). 
\end{rmk}
\begin{proof}
Note that if $a_l\stackrel{d}{\to}\infty$, $l\to\infty$, then~(\ref{Ass2}) is false.

On the other hand, if $a_l$ doesn't converge to $+\infty$ in distribution, choose $M$ and a sequence $\{l_n\}$, such that ${\bf P}(a_{l_n}\le M)>p>0$ for some positive $p$. We prove that for any $k$ 
\begin{equation}
\label{FiniteLimInfRE}
{\bf P}(\exists j\ge k: a_{j}\le M)=1.
\end{equation}
This means that $\liminf a_j \le M$ with probability one.

Let $a_{l,n}$ denote the partial sums of $a_l$. Let  $\widetilde{\tau}_0=k$ and
$$
\widetilde{\nu}_{i}=\min\{l_j: l_j>\widetilde{\tau}_i\},\quad \widetilde{\tau}_{i+1} = \min\{j>0: a_{\widetilde{\nu}_{i},j}>M\}, \ 
i\ge 0.
$$
If $\widetilde{\tau}_i=\infty$ for some $i$, then we show that~(\ref{FiniteLimInfRE}) is true. However,
$$
{\bf P}(\widetilde{\tau}_{i}<\infty) = 
\prod_{j=1}^{i-1} {\bf P}(\widetilde{\tau}_{j+1}<+\infty|\widetilde{\tau}_{j}<+\infty)\le (1-p)^i\to 0,\ i\to\infty.
$$
\end{proof}

First we apply our results to a supercritical BPRE to show that the conditions are close to the best known. In fact, our result can be applied for a supercritical model of BPRE with cooling (see \cite{Korshunov} for definitions and details).

\subsection{Supercritical BPRE}

Let $\{Z_n\}$ be supercritical BPRE, let ${\bf E}\xi_1=\mu>0$ 
and assume that ${\bf E}X e^{-\xi} \log^{1+\delta} (1+X)$ is finite for some $\delta\in (0,1)$. Let $\psi(x)=|x| \log^{\delta/2} (1+|x|)$. 

By Lemma~\ref{RegeneratingEnvLemm} we need to show that $a_1^{(\psi)}$ is finite a.s. 

By the law of large numbers $S_i>\mu i/2$ for large enough $i$. Thus, it's enough to prove that 
$$
\sum_{j=0}^{\infty} {\bf E}_{\boldsymbol\eta} U_{j} \log^{\delta/2}\left(1+\frac{U_j}{e^{\mu j/2}}\right)<+\infty.
$$
But the $j$-th term of the sum above is less or equal than
\begin{eqnarray*}
{\bf E}_{\boldsymbol\eta} \left(e^{\delta \mu j/8} \log^{\delta/2}\left(1+\frac{e^{\mu j/8}}{e^{\delta \mu j/2}}\right); U_j\le e^{\delta \mu j/8} \right)+  
{\bf E}_{\boldsymbol\eta} \left(U_j \log^{\delta/2}\left(1+U_j\right); U_j\ge e^{\delta \mu j/8}\right) \le \\
\frac{1}{e^{\delta \mu j/16}} + 
\frac{1}{j^{1+\delta/2}} {\bf E}_{\boldsymbol\eta} U_j \log^{1+\delta}(1+U_j).
\end{eqnarray*}
Thus, 
$$
\sum_{j=0}^{\infty} {\bf E}_{\boldsymbol\eta} U_{j} \log^{\delta/2}\left(1+\frac{U_j}{e^{\mu j/2}}\right)\le \sum_{j=1}^{\infty} \frac{1}{e^{\delta \mu j/16}} + 
\sum_{j=1}^{\infty} \frac{{\bf E}U_1 \log^{1+\delta}(1+U_1)}{j^{1+\delta/2}}+
\sum_{j=1}^{\infty} 
\frac{\Delta_j}{j^{1+\delta/2}},
$$
where
$$
\Delta_j = {\bf E}_{\boldsymbol{\eta}} U_j \log^{1+\delta}(1+U_j) - {\bf E}U_1 \log^{1+\delta}(1+U_1),\ j\ge 1.
$$
The sequence
$$
D_n:=\sum_{j=1}^{n} \frac{\Delta_j}{j^{1+\delta/2}}, n\ge 1,
$$
is a martingale, and
$$
{\bf E}|\Delta_j|\le {\bf E} \left({\bf E}_{\boldsymbol{\eta}} U_j \log^{1+\delta}(1+U_j)\right) + 
{\bf E}U_1 \log^{1+\delta}(1+U_1) = 2{\bf E}U_1 \log^{1+\delta}(1+U_1),\ j\ge 1,
$$
so $\sup {\bf E}|D_n|<+\infty$. Thus, $\{D_n\}$ is a.s. convergent and $a_1^{(\psi)}$ is finite a.s. So, in the particular case of a supercritical BPRE our general result gives a sufficient condition ${\bf E}X e^{-\xi} \log^{1+\delta} (1+X e^{-\xi})<+\infty$, which is a little stronger than the optimal ${\bf E}Xe^{-\xi}\log^+ X<+\infty$ from the classical work~\cite{Tanny}. 

\subsection{Critical BPRE} 

For a critical BPRE the important result is that 
\begin{equation}
\label{P+ZnConv}
\frac{Z_n}{e^{S_n}}\to W,\ {\bf P}^+-a.s.,\ n\to\infty,\quad {\bf P}^+(W>0)={\bf P}^+(Z_n>0, n\in \mathbb{N}).
\end{equation}
See Proposition 3.1 of~\cite{AGKV} for details. We don't want to repeat the conditions of article~\cite{AGKV}, thus, we just sketch the proof of Proposition 3.1  using Theorem~\ref{BPVETh}.  
The assumption~(\ref{Ass1}) is (3.5) of~\cite{AGKV}. Let check~(\ref{Ass2}). By the Tanaka decomposition (see \cite{AGKV}) for some random moments $\nu_i$, $i>0$, the sequence
$\{a_{\nu_i}\}$ corresponds to a BPRE with an i.i.d. random environment (we join together $(\nu_i-\nu_{i-1})$ generations for every $i$). 
Thus, by Lemma~\ref{RegeneratingEnvLemm} it's enough to prove that
$$
a_1=\sum_{i=0}^{\infty} \zeta_{i} e^{-S_i}<+\infty,\ {\bf P}^+-a.s.  
$$
This is Lemma 2.7 of~\cite{AGKV}. Thus, general Theorem~\ref{BPVETh} can be used instead of Proposition 3.1 of~\cite{AGKV}. 

However, the proof of Proposition 3.1 is based on the p.g.f. technique and is based on the fact that the conditional variance ${\bf E}_{\boldsymbol \eta} X^2$ is a.s. finite. Our method is much more flexible. For example, we can consider $\psi(x) = |x| \log^{\delta/2}(1+|x|)$ and check~(\ref{Ass2ForPsi}) instead of~(\ref{Ass2}). 
\begin{thm}
\label{CriticalLogTh}
Assume the Spitzer condition
$$
\frac{1}{n} \sum_{k=1}^{n} {\bf P}(S_k>0)\to \rho\in (0,1),\ n\to\infty
$$
i) Let for some $\beta\in (0,1]$ 
$$\widehat{\zeta}_i:=\ln {\bf E}_{\boldsymbol\eta} U_i^{1+\delta}, i\ge 1,
$$
be a.s. finite r.v., such that 
$$
{\bf E}\widehat{\zeta}_1<+\infty.
$$

ii) Let for some $\beta>1$, $\delta_0\in (0,1)$ 
$$\widetilde{\zeta}_i:={\bf E}_{\boldsymbol\eta} U_i \log^{\beta/\rho+\delta_0}\left(1+U_i\right), i\ge 1,
$$
be a.s. finite r.v., such that
$$
{\bf E}\widetilde{\zeta}_1^{1/(\beta-1)}<+\infty,\quad 
{\bf E}V(\xi_1) \widetilde{\zeta}_1^{\rho/(\beta-1)}<+\infty,
$$
where $V(\cdot)$ is defined by (1.4) of~\cite{AGKV}. 

Then~(\ref{P+ZnConv}) is true.
\end{thm}
\begin{rmk}
Note that in the moment conditions we can use 
$$
{\bf E}_{\boldsymbol\eta} X_1 e^{-\xi_1} \log^{\beta/\rho+\delta_0}\left(1+X_1 e^{-\xi_1}\right)
$$
instead of $\widetilde{\zeta}_1$.
\end{rmk}
\begin{proof}
The proof of~(\ref{Ass1}) is the same as before. 

i) To check~(\ref{Ass2L1+delta}) we need to prove that $a_1^{(\delta)}$ is ${\bf P}^+$-a.s. finite. Note that for any $\delta_1\in (0,1)$ and some $c(\boldsymbol\eta)$ we have $S_k\ge k^{\rho-\delta_1}$, $k>0$, for sufficiently large $k$ (see (2.22) of~\cite{AGKV}). Thus, it is enough to show that ${\bf P}^+$-a.s.
$$
\sum_{j=0}^{\infty} {\bf E}_{\boldsymbol\eta} U_{j}^{1+\delta} e^{-\delta j^{\rho-\delta_1}}<+\infty,
$$
so, it's enough to prove that
$$\widehat{\zeta}_j:= \ln {\bf E}_{\boldsymbol\eta} U_j^{1+\delta}  = O\left(j^{\rho - 2\delta_1}\right),\ j\to \infty.
$$
By (2.25) of~\cite{AGKV} we have
\begin{eqnarray*}
{\bf P}^+\left(\widehat{\zeta}_j>j^{\rho-2\delta_1}\right) \le 
{\bf P}\left(\widetilde{\zeta}_j>j^{\rho-2\delta_1}\right) + \frac{1}{j^{1-\rho-\delta}}
{\bf E}\left(V(\xi_1); \widetilde{\zeta}_j>j^{\beta+\delta/4-1}\right)\le 
\\
\frac{{\bf E}\widetilde{\zeta}_1^{1/(\beta-1)}}
{j^{1+\delta/(4(\beta-1))}}+\frac{{\bf E}V(\xi_1)\widetilde{\zeta}_1^{\rho/(\beta-1)}}{j^{1+\delta\rho/(8(\beta-1))}},\ j\ge 1.
\end{eqnarray*}

ii)  To check~(\ref{Ass2ForPsi}) we need to prove that $a_1^{(\psi)}$ is ${\bf P}^+$-a.s. finite. Again for any $\delta_1\in (0,1)$ and some $c(\boldsymbol\eta)$ we have $S_k\ge k^{\rho-\delta_1}$, $k>0$, for sufficiently large $k$. Thus, we need to check that
$$
\sum_{j=0}^{\infty} {\bf E}_{\boldsymbol\eta} U_{j} \log^{\delta/2}\left(1+\frac{U_j}{e^{j^{\rho-\delta_1}}}\right)<+\infty.
$$
Similarly to the supercritical case we obtain that 
$$
\sum_{j=0}^{\infty} {\bf E}_{\boldsymbol\eta} U_{j} \log^{\delta/2}\left(1+\frac{U_j}{e^{j^{\rho-\delta_1}}}\right)\le 
\sum_{j=1}^{\infty} \frac{1}{e^{\delta j^{\rho-\delta_1}/8}}+\sum_{j=1}^{\infty} \frac{\widetilde{\zeta}_j}{j^{\beta+\delta/2}},
$$
where $\delta$, $\delta_1$ are chosen small enough to have
$$\frac{\beta}{\rho}+\delta_0 \ge \frac{\delta}2+\frac{\beta}{\rho-\delta_1}.
$$
Thus, it's enough to prove that $\zeta_j = O(j^{\beta+\delta/4-1})$, $j\to\infty$, ${\bf P}^+$-a.s. By (2.25) of~\cite{AGKV} we have
\begin{eqnarray*}
{\bf P}^+\left(\widetilde{\zeta}_j>j^{\beta+\delta/4-1}\right) \le 
{\bf P}\left(\widetilde{\zeta}_j>j^{\beta+\delta/4-1}\right) + \frac{1}{j^{1-\rho-\delta\rho/(8(\beta-1))}}
{\bf E}\left(V(\xi_1); \widetilde{\zeta}_j>j^{\beta+\delta/4-1}\right)\le 
\\
\frac{{\bf E}\widetilde{\zeta}_1^{1/(\beta-1)}}
{j^{1+\delta/(4(\beta-1))}}+\frac{{\bf E}V(\xi_1)\widetilde{\zeta}_1^{\rho/(\beta-1)}}{j^{1+\delta\rho/(8(\beta-1))}},\ j\ge 1.
\end{eqnarray*}
Thus, 
$$\sum_{j=1}^{\infty} {\bf P}^{+}(\widetilde{\zeta}_j>j^{\beta+\delta/4-1})<+\infty$$ and by the Borel-Cantelli Lemma we have $\zeta_j = O(j^{\beta+\delta/4-1})$, $j\to\infty$, ${\bf P}^+$-a.s. This proves~(\ref{Ass2ForPsi}).
\end{proof}


\begin{thebibliography}{99}
\bibitem{VatKerst} Kersting, G., \& Vatutin, V. (2017). Discrete time branching processes in random environment. John Wiley \& Sons.
\bibitem{Jagers} Jagers, P. (1974). Galton-Watson processes in varying environments. Journal of Applied Probability, 11(1), 174-178.
\bibitem{Agresti} Agresti, A. (1975). On the extinction times of varying and random environment branching processes. Journal of Applied Probability, 12(1), 39-46.
\bibitem{Kersting} Kersting, G. (2020). A unifying approach to branching processes in a varying environment. Journal of Applied Probability, 57(1), 196-220.
\bibitem{Korshunov} Korshunov, I. D. (2025). Branching processes in random environment with freezing. Discrete Mathematics and Applications, 35(4), 235-247.
\bibitem{Pinelis} Pinelis, I. (2015). Best possible bounds of the von Bahr--Esseen type. Annals of Functional Analysis, 6(4), 1-29.
\bibitem{Orlicz} Musielak, J. (2006). Orlicz spaces and modular spaces. Springer.
\bibitem{Tanny} Tanny, D. (1988). A necessary and sufficient condition for a branching process in a random environment to grow like the product of its means. Stochastic processes and their applications, 28(1), 123-139.
\bibitem{AGKV} Afanasyev, V. I., Geiger, J., Kersting, G., \& Vatutin, V. A. (2005). Criticality for branching processes in random environment.
\end{thebibliography}
\end{document}